\documentclass[11pt]{amsart}
\usepackage{enumerate}
\usepackage{graphicx}
\usepackage{amsfonts}
\usepackage{graphicx}
\usepackage{amsmath}
\usepackage{amssymb}
\usepackage{amscd}

\newcommand{\R}{{\mathbb R}}

\newcommand{\N}{{\mathbb N}}

\renewcommand{\ni}{\noindent}

\newcommand\hsp[1]{\mbox{}\hspace{#1mm}} 
\newcommand\vsp[1]{\par \vspace{#1mm}}   
\newcommand\h{\hsp}                      
\renewcommand\v{\vsp}                    
\numberwithin{equation}{section}
\newcommand\tagr[1]{\v{-6}\mbox{}\hsp{14}\hfill#1\vsp2}    
\newcommand\minus{\!\setminus\!}

\newcommand{\supp}{\mbox{\small\bf patch}}

\newtheorem{theorem}{Theorem}
\newtheorem{lemma}{Lemma}[section]
\newtheorem{prop}[lemma]{Proposition}
\newtheorem{coro}{Corollary}

\newtheorem{remark}{Remark}
\begin{document}
\title[\Tiny Linearly repetitive Delone sets. \h4  \bf \Tiny \today]
{Delone sets with finite local complexity: Linear repetitivity versus positivity of weights}

\author{\Tiny Adnene Besbes}
\address{Institut Pr\'{e}paratoire aux \'{E}tudes d'Ing\'{e}nieur de Bizerte, Tunisia}
\email{abesbes@math.jussieu.fr }

\author{Michael Boshernitzan*}
\address{Rice University, Houston, Texas 77251, USA}
\email{michael@rice.edu }
\thanks{*The author was supported in part by NSF Grant: DMS-1102298}

\author{Daniel Lenz}
\address{Mathematisches Institut, Friedrich-Schiller Universit\"at Jena, Ernst-Abb\'{e} Platz~2, D-07743 Jena, Germany}
\email{ daniel.lenz@uni-jena.de }

\maketitle
\v3
%

\begin{abstract}
We consider Delone sets with finite local complexity. We characterize
validity of a subadditive ergodic theorem by uniform positivity of
certain weights. The latter can be considered to be an averaged
version of linear repetitivity.  In this context, we show that linear
repetitivity is equivalent to positivity of weights combined with a
certain balancedness of the shape of return patterns.
 \end{abstract}


\section{Introduction}  \label{Introduction}

Aperiodic point sets with long range order have attracted considerable
attention in recent years (see e.g. the monographs and conference  proceedings  \cite{BM,Moo,Pat,Sen}). On the one
hand, this is due to the actual discovery of physical substances, later called
quasicrystals, exhibiting such features \cite{Ni,SBGC}. On the other hand this is due
to intrinsic mathematical interest in describing the very border between
crystallinity and aperiodicity. While there is no aximatic framework for
aperiodic order yet, various types of order conditions in terms of local complexity functions  have been studied
\cite{Lag,LP,LP2}. Here, we are concerned with linear repetitivity. This condition is
given by a linear bound on the growth rate of the repetitivity function.
It has been brought forward by Lagarias/Pleasants in \cite{LP}
as a  characterization of  perfectly
ordered quasicrystals.  In fact, as shown in \cite{LP2}, among the aperiodic
repetitive Delone  sets, linearly repetitive ones have the slowest possible
growth rate of their repetitivity function.

Let us also note that independent of the work of Lagarias/Pleasants,
the analogue notion for subshifts has
received a thorough study in works of Durand \cite{Du,Du2}.

It turns out that linear repetitivity (or linear recurrence) implies validity
of a subadditive ergodic theorem \cite{DL}. Such a subadditive ergodic theorem,
is useful in application to Schr\"odinger operators and to lattice gas theory
\cite{GH,Hof1,Hof2,Len1} (see Section 9 of \cite{Len3}  for recent applications to diffraction theory as well).  Now, for subshifts it is possible to characterize validity of a
certain subadditive ergodic theorem by positivity of certain weights as shown by one of the authors in  \cite{Len}. This
positivity of weights can be considered as an averaged version of linear
repetitivity. This raises two questions:

\smallskip

Q1: Can validity of a subadditive ergodic theorem be characterized for Delone
sets in terms of positivity of certain weights?

\smallskip

Q2: What is the relationship between the positivity of weights and linear
repetitivity?

\smallskip

Our main results deal with these questions. Theorem
\ref{characterization_set} shows that positivity of certain weights is
indeed equivalent to validity of a subadditive ergodic theorem.  This
answers Question 1.  Theorem \ref{characterization_lr} characterizes
linear repetitivity in terms of positivity of these weights combined
with a  condition on the shape of return patterns.  This gives an answer to Question
2. In the case of symbolic dynamics  analogues to Theorem \ref{characterization_lr}  have been obtained earlier. In particular, there is some unpublished work of T. Monteil \cite{Mon}  characterizing linear repetitivity in terms of positivity of weights and a repulsion property as well as a more general result obtained by   Boshernitzan \cite{Bos}. Boshernitzan's result shows equivalence of linear repetitivity and positivity of weights.  Our result and its proof  are inspired by \cite{Bos}. The condition (U) appearing below is not needed in \cite{Bos} (see Remark \ref{remark-U} at the and of the paper  as well).

\section{Notation and results}\label{Notation}
Fix $N\in\N$. We consider subsets of $\R^N$. The Euclidean distance on $\R^N$
is denoted by $\varrho$. The closed ball around $p\in \R^N$ with radius $R$ is
denoted by $B_R (p)$. We write  $B_R$  for the closed ball $B_R(\bold0)$
(around the origin  $\bold0\in\R^N$).

A subset $\varLambda\subset\R^N$ is called {\em uniformly discrete}
if its {\em packing radius}
\begin{equation}\label{eq:pack}
r_{pack}(\varLambda):= \inf \big\{ \tfrac{1}{2} \, \varrho(x,y)\colon x,y\in \varLambda,
x\neq y\big\}
\end{equation}
is positive.  A subset $\varLambda$ is called {\em relatively
dense}  if its {\em covering radius}
\begin{equation}\label{eq:cov}
 r_{cov}(\varLambda):= \sup \big\{\varrho(p,\varLambda)\!\colon p\in\R^N\big\}
\end{equation}
is   finite.
A set which is both relatively dense and uniformly discrete is called a {\em Delone} set.

Let $\R^+=\{r\in\R \mid r>0\}$ stand for the set of positive real numbers.

Given a Delone set $\varLambda$, the cartesian product
\begin{equation}\label{eq:bps}
\mathcal P_{\varLambda}=\varLambda\times\R^+
\end{equation}
is refered to as {\em the collection of ball patterns on $\varLambda$}.
A ball pattern
\begin{equation}\label{eq:bp}
P=(x,R)\in \varLambda\times\R^+=\mathcal P_{\varLambda}
\end{equation}
is determined by its {\em center}\/ $c(P)=x$ and its {\em radius} $r(P)=R$.
Henceforth we often abbreviate $\mathcal P_\varLambda$ to just $\mathcal P$
when the Delone set  $\varLambda$ is clear from context.

Given a ball pattern  $P=(x,R)$, its {\em patch}\/ is defined as
the finite set
\begin{equation}\label{eq:supp}
\supp(P)=\supp(x,R)\colon=(\varLambda - x) \cap B_R\subset\R^N.
\end{equation}
Clearly, every patch contains the origin $\bold 0\in\R^N$.

A subset  $A\subset \R^N$ is called \textit{discrete} if\,
  $\sharp(A\cap B_R)<\infty$,  for all  $R>0$.
Here $\sharp(A)\!=\!\sharp A$\,  stands for the cardinality of a set $A$.

A Delone  set  $\varLambda$\,  is said to be of  {\em finite local complexity}  (FLC)
if  the difference set   $\varLambda-\varLambda=\{x-y\colon x,y\in \varLambda\}$\, is a
discrete subset of  $\R^N$. Note that $\varLambda-\varLambda$ is exactly
the union of all patches:
\[
\varLambda-\varLambda \ =\!\!\bigcup_{x\in\varLambda,\,R\in\R^+}\!\!\! \supp(x,R),
\]
and one verifies that (FLC) is equivalent to the following condition:
\begin{equation}\label{eq:flc1}
\sharp(\{\supp(x,R)\mid x\in\varLambda\})<\infty, \h2 \text{for all } R>0.
\end{equation}
\tagr{(FLC$_{1}$\!)}
\noindent (For every $R>0$  there is only a finite number of patches for
ball patterns of radius~$R$).

In fact, both (FLC)  and (FLC$_1$) are equivalent to the following condition
\begin{equation}\label{eq:flc2}
\sharp(\{\supp(x,R)\mid x\in\varLambda\})<\infty,
\quad \text{for }R=2\h{.1}r_{cov}(\varLambda),
\end{equation}
\tagr{(FLC$_{2}$\!)}
\ni(see \eqref{eq:cov} for definition of $r_{cov}(\varLambda)$).
For proofs of the equivalence of the above versions of (FLC) we refer
to \cite{Lag}.
\v1
Given a ball pattern  $P=(x,R)\in\mathcal P$,
its {\em locater set} \ $L_P$ is defined by the formula
\begin{equation}\label{eq:loc}
L_P=\big\{y\in \varLambda\colon \supp(y,R) = \supp(x,R)\big\}.
\end{equation}

A Delone set $\varLambda$  is said to be {\em repetitive}  if
the locater set  \ $L_P$ \  of any ball pattern $P\in\mathcal P$
is  relatively dense. Repetitive Delone sets $\varLambda$ satisfy (FLC), and their
locater sets  $L_P$  (of any ball pattern $P\in\mathcal P_{\varLambda}$)
form a Delone set.

In the paper we deal mostly  with repetitive Delone sets.
A Delone set $\varLambda$ is said to be {\em non-periodic} if  \
$ \varLambda - x \neq \varLambda $ \ for all $x\in\R^N$ with $\,x\neq \bold0$.  \v1

Denote by $\mathcal B$  the family of bounded subsets in $\R^{N}$\!.
Given a bounded set $Q\in \mathcal B$  and a ball pattern  $P=(x,R)\in\mathcal P_\varLambda$,
we define the set
\[
L_P(Q)=\{ y\in L_P:  B_R (y) \subset Q\},
\]
and two quantities $\sharp_P Q$, $\sharp'_P Q$ (valued in non-negative integers)
as follows.

The number $\sharp_P Q$  (called the
\textit{number of copies of $P$ in $Q$})  is defined by the formula
\begin{equation}\label{card1}
\sharp_P Q :=\sharp(L_P(Q)),
\end{equation}

The number   $\sharp'_P Q \leq\sharp_P Q $ (called the
{\em the maximal number of completely disjoint copies of $P$  in $Q$})
is defined by the formula
\begin{align}\label{eq:card2}
   \sharp'_P Q := \max  \big\{\sharp(A)\colon  &\ A\subset L_P(Q) \h3 \text{\small such that}\\
   &\varrho(x,y)>2 R, \ \text{\small for all }\, x,y\in A, x\neq y\big\}.\notag
\end{align}
(Note that the inequality $\varrho(x,y)>2R$ in the preceding line is equivalent to
the requirement that $B_R (x)\cap B_R (y) = \emptyset$).

%
\vsp2
We write  $|\cdot|$  for the Lebesgue measure in  $\R^N$\!.
Clearly  $|B_R|=R^N|B_1|$.

Next, given a ball pattern $P=(x,R)\in\mathcal P_\varLambda$, we define
two quantities, $\nu(P)$ and $\nu'(P)$, as follows.
The {\em lower density of $P$}, $\nu(P)$, is defined  by the formula
\begin{equation}\label{eq:lowdensity}
\nu(P)\colon\!\!=\liminf_{|C|\to\infty}\limits\, \tfrac{\sharp_P C\,\cdot\, |B_{R}|}{|C|}=
|B_1|\cdot\liminf_{|C|\to\infty}\limits\, \tfrac{\sharp_P C\,\cdot\, R^N}{|C|},
\end{equation}
and the {\em lower reduced density of $P$}, $\nu'(P)$,  is defined  by the formula
\begin{equation}\label{eq:lowdensity-disjoint}
\nu'(P)\colon\!\!=\liminf_{|C|\to\infty}\limits\, \tfrac{\sharp'_P C\,\cdot\, |B_{R}|}{|C|}=
|B_1|\cdot\liminf_{|C|\to\infty}\limits\, \tfrac{\sharp'_P C\,\cdot\, R^N}{|C|}.
\end{equation}
Henceforth, when writing  $|C|\!\to\!\infty$,   it is always meant that  $C$  runs over the
cubes in~$\R^N$\!.

One easily verifies that, for any Delone set $\varLambda$ and any ball pattern
$P\in\mathcal P_\varLambda$, the inequalities
\begin{align}
0\leq \nu'(P)&\leq\nu(P)<\infty,\label{ineq:nunu}\\
\nu'(P)&\leq|B_1|\,\big(\tfrac{\sqrt N}2\big)^N,\label{ineq:nup} \quad \text{and}\\
\nu'(P)&>0 \quad \text{(assuming that }\varLambda\ \text{is repetitive)}\label{ineq:nupp}
\end{align}
hold. (The inequality \eqref{ineq:nup} follows from the observation that
the diameter of a cube is smaller than $2R$ if its sidelength is smaller
than  $\tfrac{2R}{\sqrt N}$).

A Delone set $\varLambda$ is said to satisfy
{\em positivity of quasiweights} (PQ)  \ if
\begin{equation}\label{eq:PQ}
w':=\inf_{\substack{ P\in{\mathcal P} \\ r(P)\geq1}}\,\nu'(P)>0
\end{equation}
\v{-5}
\tagr{(PQ)}\v4
\noindent (see \eqref{eq:lowdensity-disjoint}).
A Delone set $\varLambda$ is said to satisfy
{\em positivity of weights} (PW)  \ if
\begin{equation}\label{eq:PW}
w:=\inf_{\substack{ P\in{\mathcal P} \\ r(P)\geq1}}\,\nu(P)>0
\end{equation}
\v{-5}
\tagr{(PW)}\v4
\noindent (see \eqref{eq:lowdensity}).
The constant $1$  (in \eqref{eq:PQ} and \eqref{eq:PW} )  is arbitrary; it can be replaced by
any other positive constant.

The following is a useful observation.
\begin{remark}  Let $\varLambda$ be a Delone set. Then $w'\leq w$, and the implications
\begin{equation}\label{eq:pqpw}
\text{($\varLambda$ satisfies PQ)}\implies\text{($\varLambda$ satisfies PW)}\implies
\text{($\varLambda$ is repetitive)}
\end{equation}
take place.
\end{remark}

  Indeed,  $w'\leq w$  follows from  \eqref{ineq:nunu}, and the implications
are straighhtforward.
\v2

Recall that $\mathcal B$  stands for the family of bounded subsets in $\R^{N}$\!.
A real valued function \mbox{$F\colon\mathcal B\to\R$}\,  is called {\em subadditive}\ if
\begin{equation}\label{ineq:subadd}
F(Q_1\cup Q_2)\leq F(Q_1) + F(Q_2),
\end{equation}
\tagr{(subadditivity)}
\ni whenever $Q_1,Q_2\in\mathcal B$ are disjoint.
\v1
Such a function $F$ is said to be {\em $\varLambda$-invariant} if
$$F(Q) = F(t +Q) \quad \mbox{\small whenever $ t+ (Q\cap \varLambda) = (t + Q)\cap \varLambda$ }.$$

A Delone set $\varLambda$ is said to satisfy a subadditive ergodic theorem (SET) if
for any subadditive invariant function $F$  the limit
\begin{equation}\label{eq:SET}
\mu(F)=\lim_{|C|\to \infty}\limits \tfrac{F(C)}{|C|}
\end{equation}
\v{-2.3}\tagr{(SET)}\v1
\noindent exists.  (Recall that, under our conventions,  $C$  runs over the cubes in $\R^N$).

\begin{theorem}\label{characterization_set}
Let $\varLambda$ be a repetitive   Delone set. Then {\em (SET)} is equivalent to {\em(PQ)}.
\end{theorem}
\begin{remark} It is possible to replace the cubes appearing in the limit of (SET) by cube like sequences. We refrain from giving details.
\end{remark}

The theorem has the following  corollary.

\begin{coro}\label{cor:cs}
If  a   Delone set $\varLambda$ satisfies (PQ), then
for every ball pattern  $P\in\mathcal P_{\varLambda}$,
the frequency \
$\lim_{|C|\to \infty}\limits \frac{\sharp_{P}C}{|C|}$  exists.
\end{coro}
\begin{proof}[\bf Proof of Corollary \ref{cor:cs}]
By Theorem \ref{characterization_set}, $\varLambda$ satisfies (SET).
The claim of Corollary~\ref{cor:cs} follows because\, $F(Q)=-\sharp_P (Q)$
 is a subadditive
function on $\mathcal B$.
\end{proof}

\begin{remark}
 It is known that  the dynamical system
associated to a  Delone set $\varLambda$  satisfying (SET) is uniquely ergodic
\cite{LMS2,LS2}.   In fact, unique ergodicity is equivalent to existence  of pattern frequencies.
Thus, Corollary \ref{cor:cs} shows that (PQ)  is also a sufficient condition for unique ergodicity
of this system.
\end{remark}

We will be concerned with further combinatorial quantities associated with Delone sets
$\varLambda$. These quantities will be discussed next.
\v2

A Delone set $\varLambda$ is called {\em linearly repetitive} (LR)  if
\v{-4}
\begin{equation}\label{eq:LR}
\sup_{\h{-1.6}\substack{P\in\mathcal P_{\varLambda} \\  r(P)\geq 1}}\limits\!
\tfrac{r_{cov}(L_P^\varLambda)}{r(P)} < \infty.
\end{equation}
\v{-5.6}
\tagr{(LR)}
\v{5.5}

A Delone set $\varLambda$  is said to satisfy the  {\em repulsion property} (RP) if
\v{-4}
\begin{equation}\label{eq:RP}
\inf_{P\in\mathcal P_{\varLambda}}\limits \ \tfrac{r_{pack}(L_P^\varLambda)}{r(P)}>0.
\end{equation}
\v{-2.5}
\tagr{(RP)}
\v2

We finally need the notion of return
pattern. For subshifts an intense study of this notion has been
carried out in work of Durand \cite{Du, Du2}. The
analogue for Delone sets (or rather tilings) has then been
investigated by Priebe \cite{Pri} (see e.\,g. \cite{PS,LS} as well).

Recall that a  subset    $\Gamma\subset\R^N$   is called discrete if
$\sharp(\Gamma\cap B_R)<\infty$  for all  $R>0$.

Given a discrete subset
$\Gamma\subset\R^n$ of cardinality  $\sharp(\Gamma)\geq2$, the Voronoi
cell of an $x\in \Gamma$  is defined by the formula:
\begin{equation}\label{eq:vor}
V_x (\Gamma):=\{ p\in R^N\hsp{-1}:\, \varrho(p,x)\leq \varrho(p,y), \mbox{ for all
$y\in \Gamma$}\}.
\end{equation}
(Recall that  $ \varrho(\cdot,\cdot)$  stands for the Euclidean distance in $\R^N$). The collection
\[
\widetilde V(\Gamma)=\{V_x(\Gamma)\mid x\in\Gamma\}
\]
is called the Voronoi partition of  $R^N$ corresponding to  $\Gamma$.

Each Voronoi cells  \mbox{$V_x=V_x(\Gamma)$} (an elements of this partition)
is a  convex subset in $\R^N$ with  \mbox{non-empty} interior containing  $x$.

The {\em distortion} $\lambda(V_x)$ of the Voronoi cell   $V_x$  is defined as the ratio
\[
\lambda(V_x)=\tfrac{r_{out}(V_x)}{r_{in} (V_x)}\in [1,\infty]
\]
of the  the {\em outer}\, and {\em inner radia}\,  of  the cell $V_x$
(denoted $r_{out}(V_x)$ and $r_{in}(V_x)$, respectively). These are defined as follows:
\begin{align*}
r_{in}(V_x)\colon\!\!&=\sup\{ R>0\colon\, B_R (x)\subset V_x\}=
\!\!\inf_{y\in \partial V_x}\limits \varrho(x,y)\!=\varrho(x, \partial V_x),\\
r_{out}
(V_x)\colon\!&=\inf\{ R>0\colon\ V_x \subset B_R (x)\}\,=\sup_{y\in \partial V_x}\limits \varrho(x,y).
\end{align*}

If  $\varLambda$ is a Delone set,  Voronoi cells\,    $V_x(\varLambda)\in\widetilde V(\varLambda)$
are easily seen to form compact polytopes  in $\R^N$\!,  and their distortions
are uniformly bounded by  $ \tfrac{r_{cov}(\varLambda)}{r_{pack}(\varLambda)}$.

By a distortion of a Delone set $\varLambda$  we mean the supremum of
distortions of its Voronoi cells:    $\lambda(\varLambda)=\sup_{x\in \varLambda}\limits  \lambda(V_x)$.

It follows  that  the distortion of any  Delone set  $\varLambda$  must be finite:
\[
1<\lambda(\varLambda)\leq  \tfrac{r_{cov}(\varLambda)}{r_{pack}(\varLambda)}<\infty.
\]
 \v1

Assume from now on that  $\varLambda$  is a fixed  repetitive  Delone set.
Then for every ball pattern  $P\in\mathcal P_{\varLambda}$,  the locater set
$L_P=L_P^\varLambda$  (see (2.1)) is itself a Delone set which determines
its own Voronoi partition
\[
\widetilde V_P=\widetilde V(L_P)=\{V_x(L_P)\mid x\in L_P\}.
\]

A Delone set $\varLambda$ is said to satisfy {\em uniformity of return
words}  (U)\,  if the distortions of all its locater sets are uniformly bounded,
i.e., if
\begin{equation}\label{eq:U}
\sup_{P\in\mathcal P_{\varLambda}}\, \lambda(L_P)<\infty.
\end{equation}
\v{-3}
\tagr{(U)}
\v2

The following is an equivalent form of the condition (U):
\[
\sup_{P\in\mathcal P_{\varLambda},\, x\in L_P}  \lambda(V_x(L_P))<\infty.
\]
\begin{theorem}\label{characterization_lr}
Let $\varLambda$ be a non-periodic Delone set. Then the following
assertions are equivalent:
\begin{itemize}
\item[(i)] The set $\varLambda$ satisfies (LR).
\item[(ii)] The set $\varLambda$ satisfies (PQ) and (U).
\item[(iii)] The  set $\varLambda$ satisfies (PW) and (U).
\end{itemize}
\end{theorem}
For definitions of the properties (LR), (PQ), (PW)
and (U) of a Delone set $\varLambda$ see \eqref{eq:LR}, \eqref{eq:PQ}, \eqref{eq:PW}   and
\eqref{eq:U}, respectively.
\begin{remark}
Note that a Delone set satisfying (LR) or (PQ) or (PW)  must be
repetitive and hence satisfy FLC (see \eqref{eq:flc1}).
\end{remark}

In Section \ref{subadditive}, we provide a proof for Theorem
\ref{characterization_set}. Section \ref{lr} deals with
Theorem~\ref{characterization_lr}.


\section{A subadditive ergodic theorem}  \label{subadditive}
This section is concerned with Theorem \ref{characterization_set}.

\begin{proof}[Proof of Theorem \ref{characterization_set}]
We have to show that the conditions (SET) and (PQ) are equivalent.
The analogue result in the setting of symbolic dynamics is given in
\cite{Len}. Here, we indicate the necessary modifications.

\smallskip

(PQ) $\Rightarrow$ (SET): The proof given in \cite{Len} is
easily carried over to show that (PQ) implies existence of the limit \
$\lim_{ |C|\to\infty}\limits  \frac{F(C)}{|C|}$  where  $C$  runs over the cubes in $\R^N$.
A~similar reasoning can also be found in the proof of (LR) $\Rightarrow$ (SET)
given in \cite{DL}.

(SET) $\Rightarrow$ (PQ): Assume to the contrary  that
(SET) holds but (PQ) fails.


Since (PQ) does not hold,
there exists a sequence $(P_n)$  of  ball patterns with  $\nu'(P_n)\to 0$
(see \eqref{eq:lowdensity-disjoint}).   By repetitivity,   $\nu'(P_n)>0$ for all  $n$
(see \eqref{ineq:nupp}).
The FLC of $\varLambda$ implies   $r(P_n)\to\infty$.

%
%
   With notation as in Section 2,  define the real valued functions
$F_n\colon\mathcal B\to\R$\,  by
\[
F_n(Q) := \sharp'_{P_n}\!(Q)\,|B_{r(P_n)}| \qquad (n\geq1).
\]
(Recall that  $\mathcal B$  stands for the family of all bounded subsets
of $\R^{n}$, and $\sharp'_{P_n}\!(Q)$  is defined as in~\eqref{eq:card2}).

Observe that each function  $-F_n$  is subadditive
and also   $\varLambda$-invariant. By  (SET),  the limits\,
$\h{-2.5}\lim_{\h2 |C|\to \infty}\limits \h{-1.5}\tfrac{F_n(C)}{|C|}$
exist  and are equal to \ $\nu (P_n)=\mu(F)$ (see \eqref{eq:SET}).

Since $\nu(P_n)\to0$,   it follows that
\[
\nu (P_n)= \mu(F_n)=\h{-3}\lim_{\h2 |C|\to \infty}\limits \h{-2}\tfrac{F_n(C)}{|C|}\to0 \qquad(n\to\infty).
\]
\v{-1}
Select now a  positive\, $\varepsilon<\frac{|B_1|}{2^N}$.  Using the relations  $\nu(P_n)\to0$  and $r(P_n)\to\infty$,
 one constructs inductively   a subsequence  $(P_{n_k})$ of  $(P_n)$  such that\\[-3mm]
\[
\tfrac1{|C|}\sum_{j=1}^k\limits  F_{n_j}(C)<\varepsilon  \qquad \text{\small and} \qquad   r(P_{n_{_{k+1}}})>r(P_{n_{k}}),
\]
for all  $k\geq1$  and all   cubes  $C$  with sidelengths\,   $s(C)\geq \frac12\,r(P_{n_{_{k+1}}})$.

Define the function $F\!\colon \mathcal B\to\R$ by
$
F(Q)\!:=\sum_{j=1}^\infty\limits F_{n_j}\!(Q).
$

Since   $r(P_n)\to\infty$, only finitely many terms in this sum do  not vanish.
As each function $-F_n$ is subadditive and $\varLambda$-invariant, so is $-F$.
Therefore,  (SET)  implies existence of the limit\,
$
\mu(F)=\lim_{|C|\to \infty}\limits \tfrac{F(C)}{|C|}.
$

For any  $k\geq1$ and for any cube $C$ with its sidelength  $s(C)$  satisfying
the inequalities\, $\tfrac{1}{2}r(P_{n_{_{k+1}} }) \leq s(C) < 2r(P_{n_{_{k+1}} })$,  we obtain  \v{-3}
\[
\tfrac{F(C)}{|C|} =  \sum_{j=1}^\infty \tfrac{F_{n_j} (C)}{|C|} = \sum_{j=1}^k \tfrac{F_{n_j} (C)}{|C|} < \varepsilon.
\]

On another hand, for  a cube $C$ with sidelength $ 2r(P_{n_k})$ and
center of mass in  $x\in L_{P_{n_k}}$,   we have  \v{-2}
\[
\tfrac{F(C)}{|C|} \geq \tfrac{F_{n_k} (C)}{|C|}  \geq \tfrac{|B_{r(P_{n_k})}|}{|C|} =
 \tfrac{|r(P_{n_k})|^N\cdot |B_1|}{|2r(P_{n_k})|^N}= \tfrac{|B_1|}{2^N}.
 \]
 \v1
It follows that \
$
\frac{|B_1|}{2^N}\leq \mu(F)\leq \varepsilon,
$
contrary to the choice of \ $\varepsilon$.

This completes the proof of Theorem \ref{characterization_set}.
\end{proof}

\section{Linear repetitive Delone sets}  \label{lr}

In this section Theorem \ref{characterization_lr} is proved.
We need several auxiliary results.
\begin{prop} \label{hp} Let $\varLambda$ be a non-periodic
Delone set.  Then (LR) implies (RP).
\end{prop}
For the definitions of the properties (LR) and (RP) see
 \eqref{eq:LR} and \eqref{eq:RP}.
\begin{proof}  For linearly repetitive non-periodic  tilings this  (and in fact a slightly
stronger result) is proven in Lemma 2.4 of  \cite{Sol}. That proof
carries easily over to Delone sets, as discussed in  Lemma 2.1 in
\cite{Len2} (see \cite{LP2} and \cite{Du2} as well).

\end{proof}

\begin{prop}\label{u}
Let $\varLambda$ be a non-periodic
Delone set.  Then, (LR) implies (U).
\end{prop}
\begin{proof}  By Propositon \ref{hp}, $\varLambda$ satisfies (RP).
Now, from (LR) and (RP) there follows existence of positive constants\, $c$\, and\, $C$\,  depending only on $\varLambda$ such that the  inequalities\,  $r_{pack}(L_{P})\geq cr(P)$\, and\, $r_{cov}(L_{P})<Cr(P)$ hold for all ball patterns $P$. These   imply that
$\frac{r_{cov}(L_{P})}{r_{pack}(L_{P})}<\tfrac Cc$.
For $x\in L_P$ we have the following inclusions for
the set  $V_x(L_P)$ (defined in \eqref{eq:vor}):
\[
B\big(x,r_{pack}(L_P)\big)\subset V_x(L_P)\subset
B\big(x,r_{cov}(L_P)\big),
\]
see Corollary 5.2 in \cite{Sen} for details.
The estimate
\[
\lambda(V_x(P))=\tfrac{r_{out}(V_x(P))}{r_{in} (V_x(P))}\leq
\tfrac{r_{cov}(L_{P})}{r_{pack}(L_{P})}\leq
\tfrac{Cr(P)}{cr(P)}<\tfrac Cc,
\]
implies $\lambda(L_{P})\leq \tfrac Cc$ which completes validation of (U)
(see \eqref{eq:U}).
\end{proof}

\begin{prop} \label{pw}
Let $\varLambda$ be a
Delone set.  Then (LR) implies (PQ).
\end{prop}
\begin{proof}
 By definition, (LR) implies that every pattern $P=(x,R)\in \mathcal P$  occurs in
 every box with sidelength $C R$.  We can now partition any sufficiently large cube  into smaller cubes of sidelength $3 C R$ up to its boundary. Each of
 these smaller cubes contains a cube of sidelength $ C R$ ``in the middle'',
 i.\,e. with distance $CR$ to the boundary. Choosing a copy of $P$ in each
  of these middle cubes, we easily obtain the statement.
\end{proof}
We will need two further lemmas before we can give the proof of the second main theorem.

\begin{lemma}\label{lem1} Let $m>n\geq1$ be integer numbers. Then \
$\sum_{k=n}^{m}\limits\frac{1}{k}\geq
\ln\!\left(\frac{m+1}{n}\right)$.
\end{lemma}
\v{-3}
\begin{proof} As the function $f(x)=\frac{1}{x}$ is decreasing on
the the interval\, $]0,+\infty[$\, we find \\
$\sum_{k=n}^{m}\limits\frac{1}{k}\geq\int_{n}^{m+1}\limits\frac{dx}{x}=\ln\!\big(\frac{m+1}{n}\big)$.
\end{proof}
\begin{lemma}\label{lem:rip}
Let $\varLambda$ be a Delone set satisfying  (PW) and (U)
and  let
\[
P=(x_{0},R)\in\varLambda\times\R^+=\mathcal P_{\varLambda}=\mathcal P
\]
be a ball pattern of radius  $R:=r(P)\geq3$. Then, for all points $x\!\in\! L_P$
(see \eqref{eq:loc}) the following inequality holds
\begin{equation}\label{eq:l45}
r_{in}\big(V_x(L_P)\big)\leq cR,
\end{equation}
where~\,$c=\max\!\left(2, 4\exp\!\big(\tfrac{6^N}{2wN}\big)\right)$,  and the constant\, $w$\,
is coming from the definition of (PW).
\end{lemma}

\begin{proof}
Set\, $d$\,  to be the distance from \ $x$ \ to the closest point  in the set
$L_P^{\Lambda}\!\setminus\! \{x\}$:
\[
d\colon\!\!=\min_{\substack{y\in L_P^\Lambda\\y\neq x}}\limits \varrho(x,y).
\]

We may assume that \ $d>4R$ \  because otherwise the inequality in  \eqref{eq:l45}
holds with $c=2$.

For $t\in\R$,  let $E(t)$ denote the greatest integer smaller than $t$. Set
\begin{align*}
d'&=E(\tfrac d3),  &&R'=E(R), \\
m&=d'-R',  &&R_k=R'+k  \quad (k\geq0).
\end{align*}

Since  $\frac d3-R>\frac R3\geq1$, it follows that\, $m\geq1$, and we have
\[
2<3\leq R'=R_0\leq R< R_1<\ldots<R_m=d'\leq\tfrac d3.
\]
Consider the following  $m$  ball patterns:
\[
P_{k}=(x, 3R_k)\in\mathcal P, \quad   1\leq k\leq m.
\]
We list the following observations (Facts 1-5):\v2
\ni {\bf Fact 1}. The following inclusions take place:
\[
x\in L_{P_i}\subset L_{P_j} \subset L_{P} \quad  (\text{for }\,1\leq j\leq i\leq m).
\]

\ni {\bf Fact 2}. For $z\in L_{P_i}$ with $1\leq i\leq k$, we have
\  $L_P\cap B_{2R_i}(z)=\{z\}$.\v1

In fact even more is true:\v2

\ni {\bf Fact 3}. For $z\in L_{P_i}$ with $1\leq i\leq k$, we have
\[
(L_P-z)\cap B_{2R_i}\stackrel{_{1}}=(L_P-x)\cap B_{2R_i}\stackrel{_{2}}=\{0\}.
\]

Indeed, the first equality $\stackrel{_{1}}=$  holds since $R=r(P)\leq R_{i}$,
and both\, \mbox{$x,z\in L_{P_i}$}.
The second equality $\stackrel{_{2}}=$ follows from the inequalities
$2R_i<3R_i\leq 3R_m<d$. \v2

\ni {\bf Fact 4}. If $z\in L_{P_i}$ and $y\in L_{P_j}$ with  $z\neq y$\,  and\,
$i,j\in\{1,2,\ldots,m\}$, then
\begin{equation}\label{eq:disj}
\Big(B_{R_i}(z)\minus B_{R_{i-1}}(z)\Big)\ \bigcap \
\Big((B_{R_j}(y)\minus B_{R_{j-1}}(y))\Big)=\emptyset.
\end{equation}

Indeed, assuming (without loss of generality) that $i\geq j$,  we have
$y\in L_{P_{j}}\subset L_{P}$ (Fact 1) and, in view of the Fact 2,
we obtain $\varrho(x,y)>2R_i$, whence the claim of Fact 4 follows.
The next observation (Fact 5) is obvious.

\v2
\ni {\bf Fact 5}. If $i\neq j$ and  $z\in L_{P_i}\cap L_{P_j}$ with
$i,j\in\{1,2,\ldots,m\}$, then the equality \eqref{eq:disj} takes place.

Given the disjointness observations \eqref{eq:disj} (Facts 4 and 5),
the inequality \vsp{-2}
\[
\sum_{k=1}^m\,
\sharp_{P_k}\big(\Phi_{d'}(C)\big)\
|B_{R_k}\!\minus B_{R_{k-1}}|\leq |C|
\]
holds for any cube  $C$. Here, for $Q\subset \R^N$ and $s>0$,  we denote by
$\Phi_s(Q)$  the set of points of $Q$ with distance at least $s$ to the
boundary of $Q$.\vsp2

Let  $C_n$  be an arbitrary sequence of cubes with $|C_n|\to\infty$. Then
\[
1\geq\sum_{k=1}^m\,
\frac{\sharp_{P_k}\big(\Phi_{d'}(C_{n})\big)}{|C_n|}\ |B_{R_k}
\!\minus B_{R_{k-1}}|, \qquad \text{for all }n.
\]

By letting  $n\to \infty$ we obtain from (PW)  (see  \eqref{eq:PW} for definition of (PW))
\begin{equation}\label{eq1} 1\geq \sum_{k=1}^{m}
\frac{w}{|P_k|}|B_{R_k}\!\minus B_{R_{k-1}}|\geq \frac{w}{3^N}\cdot\!\!
\sum_{s=R'+1}^{d'}\!\!\frac{s^N -(s-1)^N}{s^N}.
\end{equation}


Using the inequalities \
$
\frac{s^N -(s-1)^N}{s^N}>\frac{N(s-1)^{N-1}}{s^N}>\frac{N(s/2)^{N-1}}{s^N}
=\frac{N}{s\cdot2^{N-1}}\,
$
with $s\geq2$ and then Lemma \ref{lem1}, we can  rewrite \eqref{eq1}
in the form
\[1\geq\frac{w}{3^N}\cdot\frac{N}{2^ {N-1}}\cdot \h{-1}
\sum_{s=R'+1}^{d'}\frac{1}{s}\geq\frac{2wN}{6^N}\cdot
\ln\!\Big(\frac{d' + 1 }{R' + 1}\Big).
\]
It follows that \ $\frac d3\leq d'+1\leq (R'+1)\cdot\exp\!\Big(\frac{6^N}{2wN}\Big)$.
In view of the inequalities  $R'+1<R+1\leq \frac{4R}3$  (which hold for $R\geq3$),
we obtain
\[
d<4R\,\exp\!\Big(\tfrac{6^N}{2wN}\Big),
\]
completing the proof of Lemma \ref{lem:rip}.
\end{proof}

We can now turn to the
\begin{proof}[\bf Proof of Theorem {\ref{characterization_lr}}]
Through the present proof,  $\varLambda$  is assumed to be a
\mbox{non-peri}\-o\-dic Delone set,
and the properties  (PQ), (U), (LR) and FLC refer to it.\v1

We show (LR)\ $ \Longrightarrow$ ((PQ) and (U)) $\Longrightarrow$ ((PW) and (U)) $\Longrightarrow$ (LR).

\v1

The implication \ (LR)$\ \Longrightarrow\ $((PQ) and (U)) \ immediately follows
from Propositions~\ref{u} and  \ref{pw}.\v2

\v1

The implication ((PQ) and (U)) $\Longrightarrow$ ((PW) and (U)) is clear in view of
\eqref{eq:pqpw}.
\v1

In what follows, we establish linear repetivity (LR)  of a non-periodic
Delone set $\varLambda$ satifying (PW) and (U).  (Note that $\varLambda$
is repetitive and also of FLC (both properties follow from (PW))).


 Let  $t\in\R^N$ be arbitrary and let $P=(x_0,R)\in\mathcal P$  be a ball pattern
 in $\varLambda$ with  $R:=r(P)\geq 3$. Select  $x\in L_P$ such that $t\in V_x$.
By Lemma \ref{lem:rip} we have
\[
r_{in}(V_x)\leq cR
\]
and then, by $(U)$  (see \eqref{eq:U}),
\[
r_{out}(V_x)\leq  c\sigma R.
\]

As the ball of radius $ r_{out}(V_x)$ around\, $t$\,  contains $x$ we obtain
\begin{equation}\label{eq:l3}
r_{cov}(L_P)\leq c\sigma R,\ \textrm{ whenever }\ R\geq 3.
\end{equation}
Since the collection of patches corresponding to ball patterns of
radius $\leq3$ is finite in view
of the (FLC) of $\varLambda$ (see \eqref{eq:flc1}),
the (LR) follows from \eqref{eq:l3}
 (for definition of (LR) see \eqref{eq:LR}).
This completes the proof of Theorem \ref{characterization_lr}.
\end{proof}

\begin{remark} \label{remark-U}  Let us shortly comment on assumption  (U) appearing above.  It can be understood as a form of isotropy condition. From this point of view it seems reasonable that it is not necessary in the one-dimensional case. Indeed, in this case one does not need to work with the 'centers' of the Voronoi cells but can rather work with intervals to the right (or left) of the points \cite{Bos}.  One might wonder whether assumption (U) be dropped   in the higher dimensional situation as well. On the other hand,  it seems also an interesting question which restrictions are imposed on the geometry of a discrete set by  condition (U) alone.
\end{remark}

\textbf{Acknowledgements.}  This work was partially supported by the German Science Foundation (DFG). Part of this work was done while one of the authors (D.L.) was visiting Rice University. He would like to thank the department of mathematics for hospitality.

\end{document}